\documentclass[11pt,a4paper]{amsart}

\usepackage{style}
\usepackage[left=3.4cm,right=3.4cm,bottom=3.4cm,top=3.4cm]{geometry}

\usepackage{times}

\def\XLehmer{X^\dagger}
\def\sigmaLehmer{\sigma^\dagger}
\def\Enriqueslattice{\mathrm{E}_{10}}

\tikzstyle{nodal}=[circle,draw,fill=black,inner sep=0pt, minimum width=4pt]
\tikzstyle{half-fiber}=[rectangle,draw=black,thick,inner sep=0pt, minimum width=5pt, minimum height=5pt]
\tikzset{double distance = 2pt}

\title{The Enriques surface of minimal entropy}

\author{Gebhard Martin}
 \address{Mathematisches Institut \\ Universität Bonn \\ Endenicher Allee 60 \\ 53115 Bonn \\ Germany}
 \email{gmartin@math.uni-bonn.de} 

\author{Giacomo Mezzedimi}
 \address{Mathematisches Institut \\ Universität Bonn \\ Endenicher Allee 60 \\ 53115 Bonn \\ Germany}
 \email{mezzedim@math.uni-bonn.de}

\author{Davide Cesare Veniani}
 \address{IDSR \\ Universität Stuttgart \\ Pfaffenwaldring 57 \\ 70569 Stuttgart \\ Germany}
 \curraddr{Fakultät für Mathematik und Informatik \\ Universität des Saarlandes \\ Gebäude E2.4 \\ 66123 Saarbrücken}
 \email{davide.veniani@mathematik.uni-stuttgart.de}

\begin{document}

\begin{abstract}
Lehmer's number $\lambda_{10}$ is the smallest dynamical degree greater than~$1$ that can occur for an automorphism of an algebraic surface. We show that $\lambda_{10}$ cannot be realized by automorphisms of Enriques surfaces in odd characteristic, extending a result of Oguiso over the complex numbers. In contrast, we prove that in characteristic $2$ there exists a unique Enriques surface that admits an automorphism with dynamical degree~$\lambda_{10}$. We also provide explicit equations for the surface as well as for all conjugacy classes of automorphisms that realize $\lambda_{10}$.
\end{abstract}

\maketitle

\section{Introduction}

\subsection{Dynamical degrees and Lehmer's number} 
We work over an algebraically closed field~$k$ of arbitrary characteristic.
For an automorphism $\sigma \colon X \to X$ of a smooth and proper variety $X$, the \emph{algebraic entropy} $h(\sigma)$ of $\sigma$ is the natural logarithm of the spectral radius of the action of $\sigma$ on the Chow ring ${\rm CH}^{\bullet}_{\rm num}(X)$ of algebraic cycles on $X$ modulo numerical equivalence. By a result of Esnault--Srinivas \cite{EsnaultSrinivas:entropy}, the algebraic entropy of an automorphism of a smooth projective surface can be computed on its numerical group $\Num(X)$ of divisors modulo numerical equivalence.

The spectral radius of the action of $\sigma$ on $\Num(X)$ is called \emph{dynamical degree} of~$\sigma$.
By \cite[Lemma~3.1]{McMullen:K3}, this dynamical degree is either $1$, a quadratic integer, or a \emph{Salem number} of degree bounded above by the \emph{Picard rank} $\rho(X) = {\rm rk}(\Num(X))$ of $X$.
By \cite[Theorem~A.1]{McMullen:dynamics.blowups.proj.plane}, whose proof extends verbatim to our setting, the smallest possible dynamical degree greater than $1$ of a surface automorphism is Lehmer's number $\lambda_{10}$, which can be defined as the largest real root of Lehmer's polynomial
\begin{equation} \label{eq:Lehmers.polynomial}
    P_{10}(x) = x^{10} + x^9 - x^7 - x^6 -x^5 - x^4 - x^3 + x +1.
\end{equation}
We have $\lambda_{10} \approx 1.17628$ and $\log \lambda_{10} \approx 0.16236$.

As in \cite{Cantat:Thesis}, it follows easily from the classification of surfaces that the only ones admitting an automorphism $\sigma$ of positive algebraic entropy are birational to $\mathbb{P}^2$, Abelian, K3, or Enriques surfaces. 
While there are examples of rational and K3 surfaces with an automorphism $\sigma$ such that $h(\sigma) = \log \lambda_{10}$ (see \cite[Theorem 1.1]{McMullen:dynamics.blowups.proj.plane} and \cite[Theorem 7.1]{McMullen:K3}), the fact that Abelian surfaces have Picard rank smaller than $10$ implies that $h(\sigma) > \log \lambda_{10}$. If $X$ is an Enriques surface, then $\rho(X) = 10$, so, a priori, there could be examples of automorphisms $\sigma \in \Aut(X)$ with $h(\sigma) = \log \lambda_{10}$.
Over the complex numbers, the non-existence of such an automorphism was proved by Oguiso \cite[Theorem 1.2]{Oguiso:third.smallest}.

\subsection{Results} The goal of this article is to show that in characteristic $2$ there exists a unique Enriques surface with an automorphism of dynamical degree $\lambda_{10}$. More precisely, we prove the following two results in \Cref{sec: existence} and \Cref{sec: uniqueness}, respectively:

\begin{theorem} \label{thm: existence intro}
Let $X_0$ be the surface over $\mathbb{F}_{32}$ defined by Equation \eqref{eq:surface} and let $\sigma_0$ be the birational transformation of~$X_0$ defined by Equation \eqref{eq:automorphism}. Then, $X_0$ is birational to an Enriques surface $\XLehmer$ and the automorphism $\sigmaLehmer \in \Aut(\XLehmer)$ induced by $\sigma_0$ satisfies $h(\sigmaLehmer) = \log \lambda_{10}$.
\end{theorem}

\begin{theorem} \label{thm: uniqueness intro}
If \(X\) is an Enriques surface over an algebraically closed field \(k\), and \(\sigma \in \Aut(X)\) satisfies \(h(\sigma) = \log \lambda_{10}\), then \(\operatorname{char}(k) = 2\) and \(X \cong \XLehmer\).
\end{theorem}

\begin{remark} \label{rmk: intro}
The Enriques surface $\XLehmer$ is supersingular in the sense that $\Pic_{\XLehmer}^{\tau} \cong \alpha_2$, and it has the peculiar property that the canonical $\alpha_2$-torsor $Y$ over $\XLehmer$ is a normal rational surface with a single elliptic singularity. 
As explained in the proof of \Cref{thm: uniqueness main}, this $Y$ arises as the contraction of the strict transform $B_1$ of a cuspidal cubic on the blow-up $Y_1$ of $\mathbb{P}^2$ in $10$ points and $\sigmaLehmer$ arises from the automorphism of $Y_1$ of dynamical degree $\lambda_{10}$ studied by McMullen in \cite[Section~11]{McMullen:dynamics.blowups.proj.plane}. A close inspection of \cite[Section~13]{Schroeer:Enriques.normal.K3.coverings} shows that Enriques surfaces whose cover has a unique elliptic singularity form a family of dimension at least $4$, so $\XLehmer$ is distinguished even among such Enriques surfaces.
\end{remark}

To complete the picture, we also compute the field of definition, the automorphism group, and the number of conjugacy classes of automorphisms realizing Lehmer's number on $\XLehmer$. Let $W_{\operatorname{E}_{10}}$ be the Weyl group of the $\operatorname{E}_{10}$-lattice and recall that it coincides with the subgroup of $\operatorname{O}(\operatorname{E}_{10})$ of index~\(2\) preserving the two half-cones. It turns out that, in addition to the $2$-congruence subgroup
\[
    W_{\Enriqueslattice}(2) \coloneqq \Ker(W_{\operatorname{E}_{10}} \to \operatorname{O}(\Enriqueslattice/2\Enriqueslattice))
\]
which acts on every Enriques surface without $(-2)$-curves by \cite[Theorem 1.1]{Allcock} and \cite[Theorem]{DolgachevMartin}, the automorphism $\sigmaLehmer$ is enough to generate $\Aut(\XLehmer)$. More precisely, we will prove the following result as part of \Cref{thm: uniqueness main}:

\begin{theorem} \label{thm:automorphisms} 
The Enriques surface $\XLehmer$ satisfies the following properties:
\begin{enumerate}
\item It can be defined over $\mathbb{F}_2$.
\item The group $\Aut(\XLehmer)$ is an extension of $\mathbb{Z}/31\mathbb{Z}$ by $W_{\Enriqueslattice}(2)$.
\item There are ten conjugacy classes of elements of dynamical degree $\lambda_{10}$ in $\Aut(X)$.
\end{enumerate}
\end{theorem}

More precisely, we show that the ten conjugacy classes are related through Frobenius twists and taking inverses, or more explicitly by varying the choice of $\zeta$ in \Cref{sec: existence} and by taking the inverse of $\sigmaLehmer$.
As a consequence, even though the surface $\XLehmer$ can be defined over $\IF_2$, the automorphism $\sigmaLehmer$, and more generally any automorphism realizing $\lambda_{10}$, cannot.

\subsection{Strategy of proof} 

The proof of \Cref{thm: uniqueness intro} proceeds as follows. Oguiso's proof of the non-existence of a complex Enriques surface with an automorphism of dynamical degree~$\lambda_{10}$ easily extends to odd characteristic using $2$-adic cohomology. In characteristic $2$, we use canonical lifts of K3 surfaces, crystalline cohomology, and bi-conductrices to exclude the existence of $\lambda_{10}$ on Enriques surfaces whose canonical cover is non-normal or non-rational. The remaining Enriques surfaces are those of \Cref{rmk: intro}. 

To deal with this case, we first show as an application of class field theory and Gross--McMullen's theory of $P(x)$-lattices \cite{gross.mcmullen,McMullen:K3} that there is a unique conjugacy class of isometries of the lattice $\Enriqueslattice$ realizing Lehmer's number. Then, building on earlier results of Harbourne~\cite{Harbourne:blowing} and McMullen~\cite{McMullen:dynamics.blowups.proj.plane}, we prove in \Cref{lem: uniqueness} that there exists a unique rational surface $\widetilde{Y}$ of Picard rank $11$ with an automorphism $\tau$ of dynamical degree $\lambda_{10}$ and with an anticanonical cuspidal curve. The surface obtained by contracting the cuspidal curve is the K3-like covering $Y$ of $\XLehmer$. We finish the proof with a precise analysis of the action of $\tau$ on the space
$H^0(Y, T_{Y})$ of global vector fields to show that, even though $Y$ has many supersingular Enriques quotients, there is a unique one to which $\tau$ descends. In total, this proves the uniqueness of the surface~$\XLehmer$.

\subsection{Further questions} The existence of $(\XLehmer,\sigmaLehmer)$ shows that there are dynamical degrees on Enriques surfaces that can only appear in positive characteristic, answering \cite[Question~1.4]{Oguiso:enriques}. In fact, there does not seem to be an obvious constraint on the isometries realized by Enriques surface automorphisms in characteristic $2$. Hence, it makes sense to ask the following question:

\begin{question}
Can every element of $W_{\operatorname{E}_{10}}$ be realized by an automorphism of an Enriques surface in some characteristic?
\end{question}

We hope that the techniques developed in this article can be used to answer the above question and, in case the answer is negative, give a classification of all realizable isometries. 
Finally, as explained after \Cref{thm:automorphisms}, even though Lehmer's number cannot be realized by Enriques surfaces over $\mathbb{F}_2$, there exists a model of $\XLehmer$ over $\mathbb{F}_2$ such that $\sigmaLehmer$ is defined over a degree $5$ extension.

\begin{question}
Can we find explicit equations of a simple projective model of $\XLehmer$ over $\mathbb{F}_2$?
\end{question}

\subsection*{Acknowledgments} 
We thank Simon Brandhorst for suggesting to use \cite{gross.mcmullen,McMullen:K3} in the proof of \Cref{lem: E10}.

\section{Existence} \label{sec: existence}

Let $\zeta$ be a generator of $\IF_{32}^\times$ satisfying 
\[
    \zeta^5 + \zeta^2 + 1 = 0
\]
and let $k$ be an algebraic closure of $\IF_{32}$. 
Recall that an Enriques surface is a smooth and proper surface $X$ with numerically trivial canonical class $K_X$ and second Betti number $b_2(X) = 10$.

This section is dedicated to the proof of the following theorem. 

\begin{theorem} \label{thm: existence main}
In the weighted projective space $\IP_k(1,1,1,6)$, consider the surface~$X_0$ defined by
\begin{equation} \label{eq:surface}
    \begin{aligned}
        w^2 & = \zeta^{16}  x^8  y^3  z + \zeta^{12}  x^8  y  z^3 + \zeta^{20}  x^7  y^5 + \zeta^5  x^7  y^4  z + 
        \zeta^{15}  x^7  y^3  z^2 \\ & \quad + \zeta^{16}  x^7  y^2  z^3 + \zeta^{14}  x^7  y  z^4 + \zeta  x^7  z^5
        + \zeta^{17}  x^6  y^5  z +
        \zeta^6  x^6  y^3  z^3 \\ 
        & \quad + \zeta^{25}  x^6  y  z^5 + \zeta^{15}  x^5  y^7 + \zeta^{14}  x^5  y^6  z + \zeta^{27}  x^5  y^5  z^2 + 
        \zeta^{11}  x^5  y^4  z^3 \\
        & \quad + \zeta^2  x^5  y^3  z^4 + \zeta^8  x^5  y^2  z^5 + \zeta^6  x^5  y  z^6 + \zeta^{21}  x^5  z^7 + \zeta^{29}  x^4  y^7  z \\
        & \quad + \zeta^{10}  x^4  y^5  z^3 + \zeta^3  x^4  y^3  z^5 + \zeta^4  x^4  y  z^7 + \zeta^{19}  x^3  y^8  z + \zeta^3  x^3  y^7  z^2 \\ 
        & \quad +\zeta^{15}  x^3  y^6  z^3 + \zeta^{30}  x^3  y^5  z^4 + \zeta^{17}  x^3  y^4  z^5 + \zeta^5  x^3  y^3  z^6 + \zeta^3  x^3  y^2  z^7 \\
        & \quad + \zeta^{13}  x^3  y  z^8 + \zeta^4  x^2  y^7  z^3 + \zeta^4  x^2  y^5  z^5 + \zeta^{15}  x^2  y^3  z^7 + \zeta^{14}  x  y^8  z^3 \\
        & \quad + \zeta^{21}  x  y^7  z^4 + \zeta^2  x  y^6  z^5 + \zeta^{29}  x  y^5  z^6 + \zeta^{23}  x  y^4  z^7 + \zeta^{22}  x  y^3  z^8 \\
        & \quad + \zeta^{18}  y^7  z^5 + y^5  z^7.
    \end{aligned}
\end{equation}
Then, $X_0$ is birationally equivalent to an Enriques surface $\XLehmer$. Under this birational equivalence, the birational transformation $\sigma_0$ of $\IP_k(1,1,1,6)$ given by
\begin{equation} \label{eq:automorphism}
    \begin{aligned}
       \sigma_0(x:y:z:w) & = 
       (x(y+\zeta^{29}z) : 
       (y + \zeta^6 x)z :
       xz : \\
       & \quad (\zeta^{16}x^2y^2z^2)w + \zeta^{29} x^6 y^4 z^2 + \zeta^8 x^6 y^3 z^3 + \zeta^{21} x^6 y^2 z^4 \\
       & \quad + \zeta^3 x^6 y z^5 + \zeta^{11} x^5 y^5 z^2 + \zeta^{11} x^5 y^4 z^3  + \zeta x^5 y^3 z^4 \\
       & \quad + \zeta^{12} x^5 y^2 z^5 + \zeta^{13} x^5 y z^6 + \zeta^{28} x^4 y^5 z^3 + \zeta^{22} x^4 y^4 z^4 \\
       & \quad + \zeta^{23} x^4 y^3 z^5 +\zeta^{30} x^4 y^2 z^6 +\zeta^{26} x^3 y^5 z^4 +\zeta^{28} x^3 y^4 z^5 \\
       & \quad + \zeta^{16} x^3 y^3 z^6 + \zeta^{24} x^2 y^5 z^5 + \zeta^{15} x^2 y^4 z^6)
        \end{aligned}
\end{equation}
induces an automorphism $\sigmaLehmer$ of $\XLehmer$ with dynamical degree equal to Lehmer's number~$\lambda_{10}$.
\end{theorem}
\begin{proof}
Following McMullen \cite[§11]{McMullen:dynamics.blowups.proj.plane}, consider the birational transformation $f$ of~$\IP^2$ given by
\[
    f(x:y:z) = \left( x(y+\zeta^{29}z) : (y + \zeta^6 x)z : xz \right),
\]
and set 

\begin{align*}
    p_1 & = (0:0:1), \quad 
    p_2 = (1:0:0), \quad 
    p_3 = (0:1:0), \quad 
    p_4 = (\zeta^{29} : \zeta^{6} : 1), \\
    p_5 & = (\zeta^{18} : \zeta^{11} : 1), \quad
    p_6  = (\zeta^{12} : \zeta^7 : 1), \quad
    p_7 = (\zeta^{14} : \zeta^{14} : 1), \\
    p_8 & = (\zeta^{7} : \zeta^{27} : 1), \quad
    p_9 = (\zeta : \zeta^{19} : 1), \quad
    p_{10} = (\zeta^{23} : \zeta^{29} : 1).
\end{align*}
Observe that 
\begin{equation} \label{eq:p4...p1}
    f(p_i) = p_{i+1} \, \text{for all $i \in \{4,\ldots,10\}$, with $p_{11} = p_1$}, 
\end{equation}
and that there exists a unique cubic curve $B$ in $\IP^2$ passing through the points $p_1,\ldots,p_{10}$, given by
\begin{equation} \label{eq: equation of B}
    g(x,y,z)=x^2 y + \zeta^2 x^2z + \zeta^{19} xy^2 + \zeta^{13} xz^2 + \zeta^7 y^2z  + \zeta^{30} yz^2 = 0.
\end{equation}
The curve $B$ has a cusp at the point $(\zeta^{15}:\zeta^{28}:1)$, and $f$ fixes the smooth point 
\[p_0 = (\zeta^{14} : \zeta^{7} : 1)\]
of $B$. 
With notation as in \cite[Section 0.2]{CossecDolgachevLiedtke}, the projection $\pi\colon X_0 \to \IP^2$, where $X_0$ is the surface defined in the statement, is the split $\alpha_{\cL}$-torsor associated to the section $s \in H^0(\mathbb{P}^2,\cL)$ defined by the right-hand side of~\eqref{eq:surface}, where $\cL = \mathcal{O}_{\mathbb{P}^2}(6)$.
As $X_0$ has only hypersurface singularities, a straightforward computation with the Jacobian criterion shows that $X_0$ is normal with $11$ singular points lying over $p_1,\ldots,p_{10}$ and $p_0$.

Let $\beta_1 \colon Z_1 \to \IP^2$ be the blow-up of $\IP^2$ in $p_1,\ldots,p_{10}$ and denote by $E_1,\ldots,E_{10}$ the 10 exceptional divisors. 
By \cite[Theorem 11.1]{McMullen:dynamics.blowups.proj.plane}, $f$ extends to an automorphism $f_1 \colon Z_1 \to Z_1$ with dynamical degree equal to $\lambda_{10}$. 
Denoting by $B_1$ the strict transform of $B$ on $Z_1$, we compute the canonical sheaf of $Z_1$:
\[
    \omega_{Z_1} = \beta_1^*\cO_{\IP^2}(-3) \otimes \cO_{Z_1}\bigg( \sum_{i = 1}^{10} E_i \bigg) = \cO_{Z_1}(- B_1).
\]
Let now $X_1$ be the normalization of the fiber product $X_0 \times_{\IP^2} Z_1$.
For each $i \in \{1,\ldots,10\}$, let $u_i$ be a non-zero global section of $\mathcal{O}_{Z_1}(E_i)$.

\begin{claim} \label{claim:cover}
    The map $\pi_1 \colon X_1 \to Z_1$ is the split $\alpha_{\cL_1}$-torsor associated to the section  
    \[
        s_1 = \frac{\beta_1^*(s)}{\prod_{i = 1}^{10} u_i^4} \in H^0(Z_1,\mathcal{L}_1^{\otimes 2}),
    \]
    where $\mathcal{L}_1 = \mathcal{O}_{Z_1}(2B_1)$.
    Moreover, $X_1$ is smooth everywhere except at the preimage of $p_0$.
\end{claim}
\begin{proof}[Proof of the claim]
    As above, we denote by $s$ the right-hand side of \eqref{eq:surface}. 
    Because of \eqref{eq:p4...p1}, and because $\pi$ is equivariant with respect to $f$, it suffices to check the claim in a neighbourhood of $E_1$, $E_2$ and $E_3$. We do it here for $E_1$, the computations for $E_2$ and $E_3$ being analogous. 
    
    We set $z =1$ in $s$ and pull back $s$ to
    \[
        \{ ax - by = 0 \} \subseteq \mathbb{A}^2_{x,y} \times \mathbb{P}^1_{a,b}.
    \]
    In the chart $U$ given by $b = 1$, we can solve $y = a x$. After substituting, we find
    \[
        \beta_1^*(s)|_U = x^4 \widetilde s,
    \]
    with $\widetilde s \in k[a,x]$. Note that $u_1 = x$ is a local equation of $E_1$. 
    The fibre product $X_0 \times_{\IP^2} Z_1$ is given over $U$ by the equation $w^2 = \beta_1^*(s)|_U$. By a direct computation with the Jacobian criterion, one checks that the equation
    \[
    w^2 = \widetilde s
    \]
    defines a surface with no singular points over $E_1$, and is therefore normal. Hence, it is an equation for $X_1$ over $U$. Away from the other $E_i$'s , we have $s_1 = \widetilde s$ up to a unit. Therefore, locally, we obtain a 
    split $\alpha_{\cL_1}$-torsor associated to 
    \[
        \widetilde s = \frac{\beta_1^*(s)}{u_1^4}.
    \]
    Finally, observe that $s_1$ is a section of $\beta_1^*\cO_{\IP^2}(12) \otimes \cO_{Z_1}\left( \sum_{i = 1}^{10} -4E_i \right) = \mathcal{L}_1^{\otimes 2}$.
\end{proof}

We further blow up the point in $Z_1$ above $p_0$, obtaining the surface $Z_2$. Denote by $\beta_2 \colon Z_2 \to Z_1$ the blow-up morphism. 
Let $B_2$ denote the strict transform of $B_1$ and let $E_0$ be the exceptional divisor over $p_0$. 
Then,
\begin{equation} \label{eq:omega_Y''}
    \omega_{Z_2} = \cO_{Z_2}(-B_2).
\end{equation}
    
Let $X_2$ be the normalization of the fibre product $X_1 \times_{Z_1} Z_2$. A computation analogous to the one in \Cref{claim:cover} shows that $\pi_2\colon X_2 \to Z_2$ is the split $\alpha_{\cL_2}$-torsor associated to the section 
\[
    s_2 = \frac{\beta_2^*(s_1)}{u_0^2},
\]
where $u_0$ is a non-zero global section of $\mathcal{O}_{Z_2}(E_0)$, and where $\mathcal{L}_2 = \cO_{Z_2}(2B_2 + E_0)$.
Moreover, the surface $X_2$ is smooth, and by combining \cite[Proposition 0.2.20]{CossecDolgachevLiedtke} and \eqref{eq:omega_Y''}, we obtain:
\[
    \omega_{X_2} = \pi_2^*(\omega_{Z_2} \otimes \cL_2) = \pi_2^*\cO_{Z_2}(B_2 + E_0).
\]
Observe that $B_2^2 = -2$, $E_0^2 = -1$ and $B_2.E_0=1$, since $E_0$ is the exceptional divisor over a smooth point of $B_1$. Since $\pi_2$ is finite and purely inseparable, hence a universal homeomorphism by \cite[Tags 01S2, 04DC]{stacks}, each pullback $\pi_2^*B_2$ and $\pi_2^*E_0$ is either integral or twice an integral curve. From $(\pi_2^*E_0)^2 = -2$, it follows that $R_2 \coloneqq \pi_2^*E_0$ is integral; moreover, the adjunction formula yields that $p_a(R_2)=0$, so $R_2$ is smooth and rational. 
On the other hand, the adjunction formula yields $p_a(\pi_2^*B_2)<0$, so that
\[
    \pi_2^*B_2 = 2C_2
\]
for an integral curve $C_2$ with $C_2^2 = -1$. 
Moreover, $K_{X_2}.C_2 = -1$, so $C_2$ is a $(-1)$-curve. 
We can blow down $C_2$ to obtain a smooth surface $X_3$. The image $R_3$ of $R_2$ becomes a $(-1)$-curve in $X_3$, and therefore we can blow down $R_3$ to obtain a smooth surface $\XLehmer$. Since $\omega_{X_2} = \mathcal{O}_{X_2}(2C_2+R_2)$ is supported on the exceptional configuration of the morphism $X_2\to \XLehmer$, it follows that $\XLehmer$ has a trivial canonical bundle.

Recall that blowing up a smooth point increases the second Betti number by one; hence, $b_2(Z_2) = b_2(\IP^2)+11 = 12$. As $\pi_2\colon X_2 \to Z_2$ is a 
universal homeomorphism, we have $b_2(X_2) = b_2(Z_2) = 12$ by \cite[Tag 04DY]{stacks}; hence, $b_2(\XLehmer) = 10$. In particular, $\XLehmer$ is an Enriques surface.

Finally, by \cite[Theorem 11.1]{McMullen:dynamics.blowups.proj.plane} the birational transformation $f$ of $\IP^2$ defines an automorphism of $Z_2$ with dynamical degree $\lambda_{10}$. Since $\pi_2\colon X_2\to Z_2$ is a homeomorphism, the extension $\sigma_2$ to $X_2$ of the birational transformation $\sigma_0$ of~$X_0$ in the statement has dynamical degree $\lambda_{10}$ as well. The Kodaira dimension of $\XLehmer$ is~$0$, so $\sigma_2$ descends to an automorphism $\sigmaLehmer$ of $\XLehmer$ with dynamical degree~$\lambda_{10}$.
\end{proof}

\begin{remark}
We summarize in \Cref{figure:diagram} the steps of the resolution of singularities of the double cover $X_0\to \IP^2$ in \Cref{thm: existence main}. 
In order to see that the pullbacks in $X_2$ of the exceptional divisors $E_1,\ldots,E_{10}$ in $Z_2$ are rational cuspidal curves, it suffices by \cite[Tag 0BQ4]{stacks} to show that they have arithmetic genus $1$, since $\pi_2\colon X_2\to Z_2$ is a universal homeomorphism. 
Since $\pi_2^*E_i$ has square $-2$, it is integral; moreover, $K_{X_2}.\pi_2^*E_i = 2B_2.E_i = 2$. Hence, the adjunction formula yields $p_a(\pi_2^*E_i)=1$. 
We deduce that the $10$ singular points of~$X_0$ over $p_1,\ldots,p_{10}$ (in blue in \Cref{figure:diagram}) are elliptic singularities, while the singular point over $p_0$ (in red in \Cref{figure:diagram}) is an $A_1$-singularity. 
The images in $\XLehmer$ of the $10$ rational cuspidal curves in $X_2$ are rational cuspidal curves $F_1,\ldots,F_{10}$ satisfying $F_i.F_j = 2$ for $1\le i\ne j\le 10$.
\end{remark}

\begin{figure}[h!]
\centering
\begin{tikzpicture}[scale=0.5]
\draw (-2,-2.5) rectangle (3,2.5);
\node (A) at (3,1.8) [right] {\small $\mathbb{P}^2$};
\draw[scale=0.5, domain=0:2.8, smooth, variable=\x, black, thick] plot ({\x}, {sqrt((\x)^3)});
\draw[scale=0.5, domain=0:2.8, smooth, variable=\x, black, thick] plot ({\x}, {-sqrt((\x)^3)});
\node (C) at (1.2,2) [right] {\small $B$};
\filldraw[blue] (0.23,0.15) circle[radius=2.5pt];
\filldraw[blue] (0.47,0.45) circle[radius=2.5pt];
\filldraw[blue] (0.66,0.75) circle[radius=2.5pt];
\filldraw[blue] (0.82,1.05) circle[radius=2.5pt];
\filldraw[blue] (0.97,1.35) circle[radius=2.5pt];
\filldraw[blue] (0.23,-0.15) circle[radius=2.5pt];
\filldraw[blue] (0.47,-0.45) circle[radius=2.5pt];
\filldraw[blue] (0.66,-0.75) circle[radius=2.5pt];
\filldraw[blue] (0.82,-1.05) circle[radius=2.5pt];
\filldraw[blue] (0.97,-1.35) circle[radius=2.5pt];
\filldraw[red] (1.27,-2) circle[radius=2.5pt];

\draw[->] (-3.5,0) -- (-2.5,0);

\draw (-9,-2.5) rectangle (-4,2.5);
\node (Apr) at (-4,1.8) [right] {\small $Z_1$};
\draw[scale=0.5, domain=0:2.8, smooth, variable=\x, black, thick] plot ({\x-14}, {sqrt((\x)^3)});
\draw[scale=0.5, domain=0:2.8, smooth, variable=\x, black, thick] plot ({\x-14}, {-sqrt((\x)^3)});
\node (Ct) at (-5.8,2) [right] {\small $B_1$};
\draw[blue, thick] (-8.5,0.15) -- (-4.5,0.15);

\draw[blue, thick] (-8.5,0.45) -- (-4.5,0.45);
\draw[blue, thick] (-8.5,0.75) -- (-4.5,0.75);
\draw[blue, thick] (-8.5,1.05) -- (-4.5,1.05);
\draw[blue, thick] (-8.5,1.35) -- (-4.5,1.35);
\draw[blue, thick] (-8.5,-0.15) -- (-4.5,-0.15);
\draw[blue, thick] (-8.5,-0.45) -- (-4.5,-0.45);
\draw[blue, thick] (-8.5,-0.75) -- (-4.5,-0.75);
\draw[blue, thick] (-8.5,-1.05) -- (-4.5,-1.05);
\draw[blue, thick] (-8.5,-1.35) -- (-4.5,-1.35);
\filldraw[red] (-5.74,-2) circle[radius=2.5pt];

\draw[->] (-10.5,3) -- (-9.5,2);

\draw (-16,1) rectangle (-11,6);
\node (Ase) at (-11,5.3) [right] {\small $Z_2$};
\draw[scale=0.5, domain=0:2.8, smooth, variable=\x, black, thick] plot ({\x-28}, {7+sqrt((\x)^3)});
\draw[scale=0.5, domain=0:2.8, smooth, variable=\x, black, thick] plot ({\x-28}, {7-sqrt((\x)^3)});
\node (Ct) at (-12.8,5.5) [right] {\small $B_2$};
\draw[blue, thick] (-15.5,3.65) -- (-11.5,3.65);

\draw[blue, thick] (-15.5,3.95) -- (-11.5,3.95);
\draw[blue, thick] (-15.5,4.25) -- (-11.5,4.25);
\draw[blue, thick] (-15.5,4.55) -- (-11.5,4.55);
\draw[blue, thick] (-15.5,4.85) -- (-11.5,4.85);
\draw[blue, thick] (-15.5,3.35) -- (-11.5,3.35);
\draw[blue, thick] (-15.5,3.05) -- (-11.5,3.05);
\draw[blue, thick] (-15.5,2.75) -- (-11.5,2.75);
\draw[blue, thick] (-15.5,2.45) -- (-11.5,2.45);
\draw[blue, thick] (-15.5,2.15) -- (-11.5,2.15);

\node[red] (E) at (-15,1.3) [above] {\tiny $E_0$};
\draw[red, thick] (-15.5,1.5) -- (-11.5,1.5);

\draw[->] (-16.5,3) -- (-17.5,2);

\draw (-23,-2.5) rectangle (-18,2.5);
\node (Ase) at (-23,1.8) [left] {\small $Z$};

\draw[blue, thick, xshift=-20.5cm, rotate=15] (-2,0) -- (2,0);
\draw[blue, thick, xshift=-20.5cm, rotate=15 + 360/10] (-2,0) -- (2,0);
\draw[blue, thick, xshift=-20.5cm, rotate= 15 + 2*360/10] (-2,0) -- (2,0);
\draw[blue, thick, xshift=-20.5cm, rotate= 15 + 3*360/10] (-2,0) -- (2,0);
\draw[blue, thick, xshift=-20.5cm, rotate= 15 + 4*360/10] (-2,0) -- (2,0);

\filldraw[red] (-20.5,0) circle[radius=2.5pt];

\draw (-2,4.5) rectangle (3,9.5);
\node (Apr) at (3,8.8) [right] {\small $X_0$};
\draw[scale=0.5, domain=-4.5:4.5, smooth, variable=\x, black, thick] plot ({-0.8 + 1/4 * \x * \x}, {14 + \x});
\node (Cpr) at (1.6,8.8) [right] {\small $C$};
\filldraw[blue] (-0.39,7.15) circle[radius=2.5pt];
\filldraw[blue] (-0.3,7.45) circle[radius=2.5pt];
\filldraw[blue] (-0.12,7.75) circle[radius=2.5pt];
\filldraw[blue] (0.16,8.05) circle[radius=2.5pt];
\filldraw[blue] (0.51,8.35) circle[radius=2.5pt];
\filldraw[blue] (-0.39,7-0.15) circle[radius=2.5pt];
\filldraw[blue] (-0.3,7-0.45) circle[radius=2.5pt];
\filldraw[blue] (-0.12,7-0.75) circle[radius=2.5pt];
\filldraw[blue] (0.16,7-1.05) circle[radius=2.5pt];
\filldraw[blue] (0.51,7-1.35) circle[radius=2.5pt];
\filldraw[red] (1.6,7-2) circle[radius=2.5pt];

\draw[->] (0.5,4) -- (0.5,3);

\draw (-9,4.5) rectangle (-4,9.5);
\node (Atpr) at (-4,8.8) [right] {\small $X_1$};
\draw[scale=0.5, domain=-4.5:4.5, smooth, variable=\x, black, thick] plot ({-14.8 + 1/4 * \x * \x}, {14 + \x});
\node (Ctpr) at (-5.8,9) [left] {\small $C_1$};

\filldraw[red] (-7+1.6,7-2) circle[radius=2.5pt];

\draw[scale=0.5, domain=0:1, smooth, variable=\x, blue, thick] plot ({2-14+sqrt((\x)^3 * 10)}, {1.15*2 + 14+\x});
\draw[scale=0.5, domain=0:1.4, smooth, variable=\x, blue, thick] plot ({2-14-sqrt((\x)^3 * 10)}, {1.15*2 + 14+\x});

\draw[scale=0.5, domain=0:1, smooth, variable=\x, blue, thick] plot ({2-14+sqrt((\x)^3 * 10)}, {0.85*2 + 14+\x});
\draw[scale=0.5, domain=0:1.4, smooth, variable=\x, blue, thick] plot ({2-14-sqrt((\x)^3 * 10)}, {0.85*2 + 14+\x});

\draw[scale=0.5, domain=0:1, smooth, variable=\x, blue, thick] plot ({2-14+sqrt((\x)^3 * 10)}, {0.55*2 + 14+\x});
\draw[scale=0.5, domain=0:1.4, smooth, variable=\x, blue, thick] plot ({2-14-sqrt((\x)^3 * 10)}, {0.55*2 + 14+\x});

\draw[scale=0.5, domain=0:1, smooth, variable=\x, blue, thick] plot ({2-14+sqrt((\x)^3 * 10)}, {0.25*2 + 14+\x});
\draw[scale=0.5, domain=0:1.4, smooth, variable=\x, blue, thick] plot ({2-14-sqrt((\x)^3 * 10)}, {0.25*2 + 14+\x});

\draw[scale=0.5, domain=0:1, smooth, variable=\x, blue, thick] plot ({2-14+sqrt((\x)^3 * 10)}, {-0.05*2 + 14+\x});
\draw[scale=0.5, domain=0:1.4, smooth, variable=\x, blue, thick] plot ({2-14-sqrt((\x)^3 * 10)}, {-0.05*2 + 14+\x});

\draw[scale=0.5, domain=0:1, smooth, variable=\x, blue, thick] plot ({2*1-14+sqrt((\x)^3 * 10)}, {-1.55*2 + 14+\x});
\draw[scale=0.5, domain=0:1.4, smooth, variable=\x, blue, thick] plot ({2*1-14-sqrt((\x)^3 * 10)}, {-1.55*2 + 14+\x});

\draw[scale=0.5, domain=0:1, smooth, variable=\x, blue, thick] plot ({2*1-14+sqrt((\x)^3 * 10)}, {-1.25*2 + 14+\x});
\draw[scale=0.5, domain=0:1.4, smooth, variable=\x, blue, thick] plot ({2*1-14-sqrt((\x)^3 * 10)}, {-1.25*2 + 14+\x});

\draw[scale=0.5, domain=0:1, smooth, variable=\x, blue, thick] plot ({2-14+sqrt((\x)^3 * 10)}, {-0.95*2 + 14+\x});
\draw[scale=0.5, domain=0:1.4, smooth, variable=\x, blue, thick] plot ({2-14-sqrt((\x)^3 * 10)}, {-0.95*2 + 14+\x});

\draw[scale=0.5, domain=0:1, smooth, variable=\x, blue, thick] plot ({2-14+sqrt((\x)^3 * 10)}, {-0.65*2 + 14+\x});
\draw[scale=0.5, domain=0:1.4, smooth, variable=\x, blue, thick] plot ({2-14-sqrt((\x)^3 * 10)}, {-0.65*2 + 14+\x});

\draw[scale=0.5, domain=0:1, smooth, variable=\x, blue, thick] plot ({2-14+sqrt((\x)^3 * 10)}, {-0.35*2 + 14+\x});
\draw[scale=0.5, domain=0:1.4, smooth, variable=\x, blue, thick] plot ({2-14-sqrt((\x)^3 * 10)}, {-0.35*2 + 14+\x});

\draw[->] (-6.5,4) -- (-6.5,3);
\draw[->] (-3.5,7) -- (-2.5,7);

\draw (-16,8) rectangle (-11,13);
\node (A2) at (-11,12.3) [right] {\small $X_2$};
\draw[scale=0.5, domain=-4.5:4.5, smooth, variable=\x, black, thick] plot ({-28.8 + 1/4 * \x * \x}, {21 + \x});
\node (C2) at (-12.8,12.5) [left] {\small $C_2$};

\draw[red, thick] (-15.5,8.5) -- (-11.5,8.5);
\node[red] (D) at (-15,8.4) [above] {\small $R_2$};

\draw[scale=0.5, domain=0:1, smooth, variable=\x, blue, thick] plot ({2-28+sqrt((\x)^3 * 10)}, {1.15*2-7 + 28+\x});
\draw[scale=0.5, domain=0:1.4, smooth, variable=\x, blue, thick] plot ({2-28-sqrt((\x)^3 * 10)}, {1.15*2-7 + 28+\x});

\draw[scale=0.5, domain=0:1, smooth, variable=\x, blue, thick] plot ({2-28+sqrt((\x)^3 * 10)}, {0.85*2-7 + 28+\x});
\draw[scale=0.5, domain=0:1.4, smooth, variable=\x, blue, thick] plot ({2-28-sqrt((\x)^3 * 10)}, {0.85*2-7 + 28+\x});

\draw[scale=0.5, domain=0:1, smooth, variable=\x, blue, thick] plot ({2-28+sqrt((\x)^3 * 10)}, {0.55*2-7 + 28+\x});
\draw[scale=0.5, domain=0:1.4, smooth, variable=\x, blue, thick] plot ({2-28-sqrt((\x)^3 * 10)}, {0.55*2-7 + 28+\x});

\draw[scale=0.5, domain=0:1, smooth, variable=\x, blue, thick] plot ({2-28+sqrt((\x)^3 * 10)}, {0.25*2-7 + 28+\x});
\draw[scale=0.5, domain=0:1.4, smooth, variable=\x, blue, thick] plot ({2-28-sqrt((\x)^3 * 10)}, {0.25*2-7 + 28+\x});

\draw[scale=0.5, domain=0:1, smooth, variable=\x, blue, thick] plot ({2-28+sqrt((\x)^3 * 10)}, {-0.05*2-7 + 28+\x});
\draw[scale=0.5, domain=0:1.4, smooth, variable=\x, blue, thick] plot ({2-28-sqrt((\x)^3 * 10)}, {-0.05*2-7 + 28+\x});

\draw[scale=0.5, domain=0:1, smooth, variable=\x, blue, thick] plot ({2*1-28+sqrt((\x)^3 * 10)}, {-1.55*2-7 + 28+\x});
\draw[scale=0.5, domain=0:1.4, smooth, variable=\x, blue, thick] plot ({2*1-28-sqrt((\x)^3 * 10)}, {-1.55*2-7 + 28+\x});

\draw[scale=0.5, domain=0:1, smooth, variable=\x, blue, thick] plot ({2*1-28+sqrt((\x)^3 * 10)}, {-1.25*2-7 + 28+\x});
\draw[scale=0.5, domain=0:1.4, smooth, variable=\x, blue, thick] plot ({2*1-28-sqrt((\x)^3 * 10)}, {-1.25*2-7 + 28+\x});

\draw[scale=0.5, domain=0:1, smooth, variable=\x, blue, thick] plot ({2-28+sqrt((\x)^3 * 10)}, {-0.95*2-7 + 28+\x});
\draw[scale=0.5, domain=0:1.4, smooth, variable=\x, blue, thick] plot ({2-28-sqrt((\x)^3 * 10)}, {-0.95*2-7 + 28+\x});

\draw[scale=0.5, domain=0:1, smooth, variable=\x, blue, thick] plot ({2-28+sqrt((\x)^3 * 10)}, {-0.65*2-7 + 28+\x});
\draw[scale=0.5, domain=0:1.4, smooth, variable=\x, blue, thick] plot ({2-28-sqrt((\x)^3 * 10)}, {-0.65*2-7 + 28+\x});

\draw[scale=0.5, domain=0:1, smooth, variable=\x, blue, thick] plot ({2-28+sqrt((\x)^3 * 10)}, {-0.35*2-7 + 28+\x});
\draw[scale=0.5, domain=0:1.4, smooth, variable=\x, blue, thick] plot ({2-28-sqrt((\x)^3 * 10)}, {-0.35*2-7 + 28+\x});

\draw[->] (-10.5,10) -- (-9.5,9);
\draw[->] (-16.5,10) -- (-17.5,9);
\draw[->] (-13.5,7.5) -- (-13.5,6.5);

\draw (-23,4.5) rectangle (-18,9.5);
\node (X) at (-23,8.8) [left] {\small $\XLehmer$};

\draw[->] (-20.5,4) -- (-20.5,3);

\foreach \k in {1,...,10} {
  \draw[blue, thick,
        xshift=-20.5cm,
        yshift=7cm,
        rotate= 15 + \k*360/10]
    plot[domain=-1:1, smooth, variable=\t]
      ({0.4*\t*\t - 0.4}, {\t*\t*\t - 1});
}

\filldraw[red] (-20.5,7) circle[radius=2.5pt];
\node (X) at (-20.7,7.02) [left,red] {\tiny $2$};

\end{tikzpicture}

\caption{The resolution of singularities of the surface $X_0$ in \Cref{thm: existence main}.}
\label{figure:diagram}
\end{figure}
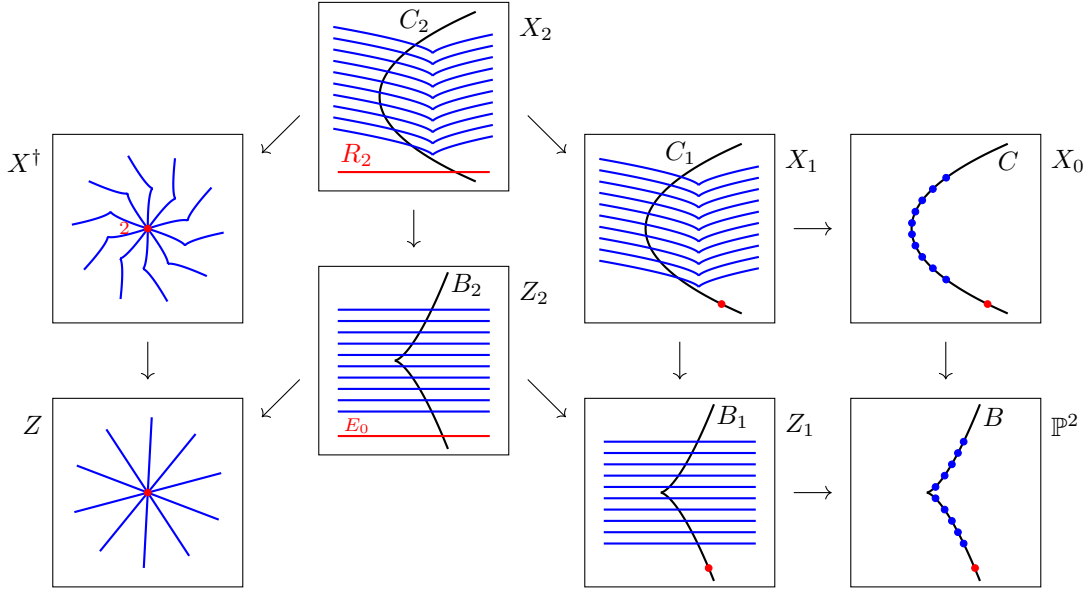

\begin{remark} \label{rmk: Giacomos remark}
The normal surface $Z$ in \Cref{figure:diagram} is the contraction of the curves $E_0$ and $B_2$ in $Z_2$. Note that such a contraction exists by \cite[Theorem 2.9]{artin}. 
As $\pi_2 \colon X_2 \to Z_2$ is an $\alpha_{\cL_2}$-torsor, where $\cL_2$ is the line bundle associated to a divisor supported on the exceptional locus of $Z_2 \to Z$, the induced map $\pi \colon \XLehmer \to Z$ is an $\alpha_2$-torsor over the complement $U = Z - \{z\}$ of the singular point $z$ of~$Z$. There is thus an $\alpha_2$-action on $\XLehmer$ that is free outside $\pi^{-1}(z)$ and such that the quotient map $\XLehmer \to \XLehmer/\alpha_2$ coincides with $\pi$ over $U$. Since $Z$ is normal and $\pi$ is $\alpha_2$-invariant, we deduce the existence of a compatible isomorphism $\XLehmer/\alpha_2 \cong Z$.

In simpler terms, this means that the double cover $X_0 \to \IP^2$ in \Cref{thm: existence main} is birationally equivalent to the quotient $Z$ of the Enriques surface $\XLehmer$ by the unique (up to scalar multiplication) regular $2$-closed derivation $D$ in $H^0(\XLehmer,T_{\XLehmer})=k$. In the coordinates of~$X_0$, we can write down this derivation as 
\[
    D = g^2 \partial_w,
\]
where $g=g(x,y,z)$ is the defining equation of the cuspidal cubic curve $B$ in \Cref{eq: equation of B}. In order to see this, we first show that 
\begin{equation} \label{eq: sigma0*D*sigma0^-1}
    \sigma_0 D \sigma_0^{-1} = \zeta^8D,
\end{equation}
where $\sigma_0$ is the birational transformation of~$X_0$ defined in \Cref{thm: existence main}. A straightforward computation shows that
\[
    \sigma_0^{-1}(x:y:z:w) = ((x+az)z:(x+az)(y+bz):(y+bz)z: \zeta^{15}\alpha w + \eta)
\]
where $a=\zeta^{29}$, $b=\zeta^6$, $\alpha = (x+az)^2(y+bz)^2z^2$ and $\eta$ is a polynomial in $x,y,z$. In order to check \eqref{eq: sigma0*D*sigma0^-1}, it suffices to work in the affine chart $z\ne 0$. Moreover, $D$ maps the subring $k\left( \frac{x}{z}, \frac{y}{z}\right)$ to $0$, so it suffices to check equality on the element $\frac{w}{z^6}$. 
Clearly, $D\left( \frac{w}{z^6} \right) = \frac{g^2}{z^6}$. 
On the other hand, one easily computes that $\sigma_0(g)=(\zeta^{12}xyz) g$, so 
\begin{equation*}
\begin{aligned}
\sigma_0 D\sigma_0^{-1}\left( \frac{w}{z^6} \right) 
    & = \sigma_0 D\left( \frac{\zeta^{15}\alpha w + \eta}{(y+bz)^6z^6} \right) 
    = \sigma_0\left( \frac{\zeta^{15}g^2(x+az)^2(y+bz)^2z^2}{(y+bz)^6z^6} \right) \\
    & = \frac{\zeta^8g^2x^2y^2z^2(xy)^2(yz)^2(xz)^2}{(yz)^6(xz)^6} 
    = \zeta^8 \frac{g^2}{z^6}.
\end{aligned}
\end{equation*}
The derivation $D$ extends to a $2$-closed derivation on $\XLehmer$ (which we also denote $D$), that is regular away from the exceptional divisors $E_1,\ldots,E_{10}$ over the elliptic singularities of~$X_0$. 
If $D$ had a pole along $E_i$, it would have a pole along $(\sigmaLehmer)^n(E_i)$ for every $n\ge 0$, by the fact that the automorphism $\sigmaLehmer$ of $\XLehmer$ normalizes $D$. This is a contradiction by \Cref{eq:p4...p1}, since $\sigmaLehmer(E_1)$ is not one of the $E_i$. 
In particular, up to scalar multiplication, $D$ is the unique regular $2$-closed derivation in $H^0(\XLehmer,T_{\XLehmer})$.
\end{remark}

\section{Uniqueness} \label{sec: uniqueness}
In this section, we establish the uniqueness result stated in \Cref{thm: uniqueness intro}. Throughout, $X$ denotes an Enriques surface over an algebraically closed field $k$.

First, we extend Oguiso's argument of \cite[Theorem 1.2]{Oguiso:third.smallest} to odd characteristic by using $2$-adic cohomology instead of singular cohomology. 
Recall that, in characteristic different from $2$, the universal \'etale cover $\pi \colon Y \to X$ is an étale double cover by a K3 surface $Y$. 

\begin{proposition} \label{prop: charneq2}
If $\operatorname{char}(k) \neq 2$, then $h(\sigma) > \log \lambda_{10}$ for all $\sigma \in \Aut(X)$.
\end{proposition}
\begin{proof}
Assume there exists $\sigma \in \Aut(X)$ with $h(\sigma) = \log \lambda_{10}$ and let $\pi \colon Y \to X$ be the K3 cover of $X$. Let $\tau \in \Aut(Y)$ be a lift of $\sigma$. As in \cite[Section 4]{Oguiso:third.smallest}, the pullback $\pi^*$ identifies $\Num(X) \cong {\Enriqueslattice}$ with a primitive sublattice $L \subseteq \Pic(Y)$ isometric to ${\Enriqueslattice}(2)$. Moreover, the action of $\sigma$ on $\Num(X)/2\Num(X)$ is identified via $\pi^*$ with the action of the lift $\tau$ on the discriminant group $A_L$ of $L$.

Since $\sigma$ acts on $\Num(X)$ via an isometry whose characteristic polynomial is Lehmer's polynomial $P_{10}$ \eqref{eq:Lehmers.polynomial},
the characteristic polynomial of the automorphism $\bar{\tau}$ of $A_L$ induced by $\tau$ is the reduction of $P_{10}$ modulo~$2$, which can be factorized into irreducible factors as 
\begin{equation}\label{eq:Lehmer.mod2}
(x^5+x^3+x^2+x+1)(x^5+x^4+x^3+x^2+1),
\end{equation}
see \cite[Lemma 4.3]{Oguiso:third.smallest}. Note that the roots of this polynomial are pairwise distinct primitive $31$-st roots of unity, hence $\bar\tau$ is diagonalizable over $\IF_{32}$ and ${\rm ord}(\bar{\tau}) = 31$. 

Now, set $M = L^{\perp} \subseteq \Pic(Y)$, so that $L \oplus M \subseteq \Pic(Y)$ is a finite index sublattice with both $L$ and $M$ preserved by $\tau$, and let $T_2 \coloneqq (\Pic(Y)_{\mathbb{Z}_2})^\perp \subseteq H^2_{\text{\'et}}(Y,\mathbb{Z}_2)$ be the $2$-adic transcendental lattice of $X$, so that 
$$L_{\mathbb{Z}_2} \oplus M_{\mathbb{Z}_2} \oplus T_2 \subseteq\Pic(Y)_{\mathbb{Z}_2} \oplus T_2 \subseteq H^2_{\text{\'et}}(Y,\mathbb{Z}_2)$$
is a finite index $\mathbb{Z}_2$-sublattice, again with $L_{\mathbb{Z}_2}$, $M_{\mathbb{Z}_2}$ and $T_2$ preserved by~$\tau$.
By \cite[Corollary~1.2]{EsnaultSrinivas:entropy}, the isometry $\tau|_{M_{\mathbb{Z}_2} \oplus T_2}$ has finite order. 
On the other hand, since $H^2_{\text{\'et}}(Y,\mathbb{Z}_2)$ is unimodular by Poincaré duality and $\overline{\tau}$ has order $31$, the order of $\tau|_{M_{\mathbb{Z}_2} \oplus T_2}$ must be divisible by $31$.
Since ${\rm rk}~M = \rho(Y) - 10 \leq 12$, the order of $\tau|_{M_{\mathbb{Z}_2}}$ is not divisible by $31$,
hence the order of $\tau|_{T_2}$ is. In particular, $T_2 \neq 0$, so $Y$ has finite height. Then, $31 \mid {\rm ord}(\tau|_{T_2})$ is impossible by \cite[Proposition 3.7, Remark 3.8]{Jang:Frobenius}.
\end{proof}

Thus, we can focus on the case $\operatorname{char}(k) = 2$. Here, there are three types of Enriques surfaces, distinguished by the torsion component $\Pic_X^{\tau}$ of their Picard scheme. We refer the reader to \cite[Chapter 1]{CossecDolgachevLiedtke} for an introduction to Enriques surfaces in characteristic $2$. We have $\Pic_X^{\tau} \in \{\mu_2,\mathbb{Z}/2\mathbb{Z},\alpha_2\}$ and $X$ is called \emph{ordinary}, \emph{classical}, or \emph{supersingular}, respectively.
Let $G \coloneqq \Pic_X^{\tau}$ and let $G^D \coloneqq \operatorname{Hom}(G,\mathbb{G}_m)$ be its Cartier dual. 
By \cite[Proposition~(6.2.1)]{Raynaud:picard}, there is a $G^D$-torsor $\pi \colon Y \to X$ which has the universal property that if $H$ is any finite commutative group scheme and $\pi' \colon Y' \to X$ is an $H$-torsor, then $\pi$ factors uniquely equivariantly through $\pi'$. 
In particular, every automorphism of $X$ lifts to $Y$.

In case $X$ is ordinary, the morphism $\pi \colon Y \to X$ is étale and $Y$ is a K3 surface. Using the theory of canonical lifts to characteristic $0$, this case can be excluded quickly:

\begin{proposition} \label{prop: char2ordinary}
If $\operatorname{char}(k) = 2$ and $X$ is ordinary, then $h(\sigma) > \log \lambda_{10}$ for all $\sigma \in \Aut(X)$.
\end{proposition}
\begin{proof}
Let $\sigma \in \Aut(X)$ and $\pi \colon Y \to X$ the K3 cover of $X$. Let $\tau \in \Aut(Y)$ be a lift of $\sigma$. Since $X$ is ordinary, so is $Y$ by \cite[Theorem~2.7]{Crew}. By \cite[Theorem~4.11]{Srivastava}, we can lift $X,Y,\tau,$ and $\sigma$ compatibly to characteristic $0$, so the statement follows from \cite[Theorem~1.2]{Oguiso:third.smallest}.
\end{proof}

If $X$ is not ordinary, then $\pi \colon Y \to X$ is purely inseparable and $h^0(X,\Omega_X^1) = 1$. The surface $Y$ is \emph{K3-like} in the sense that it is integral and Gorenstein with $\omega_Y \cong \mathcal{O}_Y$ and $h^1(Y,\mathcal{O}_Y) = 0$, but it is always singular. There are the following three possibilities for the shape of the singularities of $Y$, see \cite[Theorem~1.3.5]{CossecDolgachevLiedtke}, \cite[Theorem~14.1]{Schroeer:Enriques.normal.K3.coverings}, and \cite[Theorem~1.4]{Matsumoto}:
\begin{enumerate}
\item[(A)] The surface $Y$ is not normal. In this case, the image of the non-normal locus of $Y$ is the support of the \emph{bi-conductrix} $B$, which is the divisorial part of the zero locus of a non-zero global $1$-form $\omega \in H^0(X,\Omega_X^1)$. The divisor $B$ is a sum of $(-2)$-curves.
\item[(B)] The surface $Y$ is normal and has only rational double point singularities. In this case, the minimal resolution $\widetilde{Y}$ of $Y$ is a supersingular K3 surface.
\item[(C)] The surface $Y$ is normal and has a unique isolated singularity formally isomorphic to the elliptic double point $k \llbracket x,y,z \rrbracket/(z^2 + x^3 + y^7)$.
\end{enumerate}

Thus, our goal is to show that, in cases (A) and (B), Lehmer's number is not attained as dynamical degree, while in case (C) it exists on a unique Enriques surface. First, we observe that a non-empty bi-conductrix puts constraints on dynamical degrees:

\begin{proposition} \label{prop: char2nonnormal}
If $\operatorname{char}(k) = 2$, $X$ is not ordinary, and the K3-cover $\pi \colon Y \to X$ of $X$ is not normal, then $h(\sigma) > \log\lambda_{10}$ for all $\sigma \in \Aut(X)$.
\end{proposition}

\begin{proof}
Let $L \subseteq \Num(X)$ be the sublattice spanned by the components of the bi-conductrix $B$. Since $B$ is non-empty, $L$ is non-trivial. Any $\sigma \in \Aut(X)$ preserves the decomposition of the sublattice $L \oplus L^{\perp} \subseteq \Num(X)$, hence the characteristic polynomial of $\sigma$ cannot be irreducible. Since $\lambda_{10}$ has degree $10$, we conclude that $h(\sigma) \neq \log \lambda_{10}$.
\end{proof}

In the case where $Y$ is normal and its minimal resolution $\widetilde{Y}$ is a supersingular K3 surface, we want to mimic the argument of \Cref{prop: charneq2}. There, we compared the unimodularity of the second cohomology group with our knowledge of the action of $\sigma$ on $\Num(X)/2\Num(X)$. To do this in characteristic $2$, we need to work with crystalline cohomology.

\begin{proposition} \label{prop: char2normal}
If $\operatorname{char}(k) = 2$, $X$ is not ordinary, and the K3-cover $\pi \colon Y \to X$ of $X$ is normal with only rational double points as singularities, then $h(\sigma) > \log \lambda_{10}$ for all $\sigma \in \Aut(X)$.
\end{proposition}

\begin{proof}
Assume there exists $\sigma \in \Aut(X)$ with $h(\sigma) = \log \lambda_{10}$.
Let $\gamma \colon \widetilde{Y} \to Y$ be the minimal resolution of $Y$, let $\widetilde{\pi} = \pi \circ \gamma$, and let $\tau$ be the lift of $\sigma$ to $\widetilde{Y}$. By \cite[Lemma 6.6]{Ekedahl.Hyland.Shepherd-Barron} and since $\gamma$ is a composition of blow-ups in closed points, we have 
\[
    L \coloneqq \Enriqueslattice(2) \cong \widetilde{\pi}^* \Pic(X) = \gamma^* \Pic(Y) \subseteq \Pic(\widetilde{Y})
\]
By \cite[5.21.4]{Illusie} and \cite[§6, first paragraph]{Ekedahl.Hyland.Shepherd-Barron}, we have a $W$-sublattice of finite index
\[
    N_W \coloneqq L_W \oplus M_W \subseteq \Pic(\widetilde{Y})_W \subseteq H^2_{\mathrm{\mathrm{cris}}}(\widetilde{Y}/W),
\]
where $M$ is the index two overlattice of $A_1^{12}$ obtained by adjoining $\frac{1}{2}(v_1+\ldots+v_{12})$.
The automorphism $\tau$ preserves $L_W$ and the saturation of $M_W$ in $H^2_{\mathrm{\mathrm{cris}}}(\widetilde{Y}/W)$.
By Poincaré duality, the $W$-lattice $H^2_{\mathrm{\mathrm{cris}}}(\widetilde{Y}/W)$ is unimodular.

Now, denote by $A_{N_W}, A_{L_W}$ and $A_{M_W}$ the discriminant groups of $N_W, L_W$, and $M_W$, respectively. Since $N$, $L$, and $M$ are $2$-elementary lattices with discriminant $2^{20},2^{10}$ and~$2^{10}$, respectively, these discriminant groups are $k$-vector spaces of dimension $20$, $10$, and~$10$. Since $H^2_{\mathrm{cris}}(\widetilde{Y}/W)$ is unimodular, the cokernel $V$ of the inclusion $N_W \subseteq H^2_{\mathrm{cris}}(\widetilde{Y}/W)$ is a maximal isotropic $k$-subspace of dimension $10$ of $A_{N_W}$.
Since $L_W$ and $M_W$ glue to a unimodular lattice, there is an anti-isometry $\varphi \colon A_{L_W} \to A_{M_W}$ that is an isomorphism at the level of $k$-vector spaces. Denote by $\pi_L$ and $\pi_M$ the two projections from $A_{N_W}$ to $A_{L_W}$ and~$A_{M_W}$, respectively. If $\overline{L_W}$ and $\overline{M_W}$ denote the saturations of $L_W$ and $M_W$ in $H^2_{\mathrm{cris}}(\widetilde{Y}/W)$, respectively, we have a sequence of inclusions
$$L_W \oplus M_W \subseteq \overline{L_W}\oplus \overline{M_W} \subseteq H^2_{\mathrm{cris}}(\widetilde{Y}/W).$$
Moreover, the two saturations induce subgroups $V_L$ and $V_M$ of $A_{L_W}$ and $A_{M_W}$, such that $V_L\subseteq \pi_L(V)$, $V_M\subseteq \pi_M(V)$ and $\varphi(V_L)=V_M$. Observe that $V_L \ne 0$ (or equivalently $V_M\ne 0$) if and only if $\pi_L|_V$ (or $\pi_M|_V$) is injective, or equivalently an isomorphism.

We repeat the previous considerations for $\overline{N_W} \coloneqq \overline{L_W}\oplus \overline{M_W}$: we denote by $A_{\overline{L_W}}$ and $A_{\overline{M_W}}$ the discriminant groups of $\overline{L_W}$ and $\overline{M_W}$, by $\overline{V}\subseteq A_{\overline{N_W}}$ the isotropic subgroup corresponding to the inclusion $\overline{N_W}\hookrightarrow H^2_{\mathrm{cris}}(\widetilde{Y}/W)$, and by $\overline{\pi}_L$ and $\overline{\pi}_M$ the two projections from $A_{\overline{N_W}}$ to $A_{\overline{L_W}}$ and $A_{\overline{M_W}}$. 
By construction, the two restrictions $\overline{\pi}_L|_{\overline{V}}$ and $\overline{\pi}_M|_{\overline{V}}$ are injective, and are therefore isomorphisms. 
Since $\tau$ preserves $L_W$ and $\overline{M_W}$, we deduce that the isomorphism of $k$-vector spaces $\overline{\varphi} \colon A_{\overline{L_W}} \to A_{\overline{M_W}}$ commutes with $\tau$. 
However, since ${\rm rk} ~M = 12$ and $M$ is negative definite, the automorphism $\tau$ acts on $M$ with finite order coprime to $31$, so up to replacing $\sigma$ with $\sigma^a$, with $\gcd(a,31)=1$, we may assume that $\tau$ acts trivially on $M$, and thus on $A_{\overline{M_W}}$. 
Consequently, $\tau$ acts trivially on $A_{\overline{L_W}} \subseteq A_{L_W}/V_L$. 
Since $\tau$ preserves the subspace $V_L$ of $A_{L_W}$, it follows that $A_{\overline{L_W}}$ lifts to a subspace of $A_{L_W}$ over which $\tau$ acts as the identity. 
However, the action of $\tau$ on $A_{L_W}$ can be diagonalized with eigenvalues distinct roots of unity of order $31$, since the same is true for the action of $\sigma$ on the discriminant group of$L$, which has characteristic polynomial as in \eqref{eq:Lehmer.mod2}. 
Therefore, the discriminant group $A_{\overline{L_W}}$ is trivial. It follows that $\overline{L_W}$ (and thus $\overline{M_W}$) are unimodular, and therefore $V_M = \pi_M(V)$ is an isotropic subspace of $A_{M_W}$ of maximal dimension $5$.
However, it follows from \cite[Lemma~9.3.(1)]{Ekedahl.Hyland.Shepherd-Barron} that the subspace $\pi_M(V)$ of $A_{M_W}$ is $\IF_2$-rational, that is, it is the base change to $k$ of a subgroup $V_{\IF_2}\subseteq A_{M\otimes \IZ_2}$, which therefore is isotropic of maximal dimension $5$.
This is a contradiction, because $V_{\IF_2}$ would induce a unimodular overlattice of $A_1^{12}\otimes \IZ_2$, which does not exist by \cite[Theorem~3.6.2]{nikulin}.
\end{proof}

It remains to study the case where the canonical cover $\pi \colon Y \to X$ is a normal rational surface.
We will need the following lattice-theoretical uniqueness result:

\begin{lemma} \label{lem: E10}
There is a unique conjugacy class of isometries in ${\rm O}(\Enriqueslattice)$ with characteristic polynomial $P_{10}$ (\ref{eq:Lehmers.polynomial}).
\end{lemma}
\begin{proof}
In McMullen's notation \cite[§5]{mcmullen.k3.minimum}, it suffices to show that there exists a unique unimodular $P_{10}(x)$-lattice of signature $(1,9)$ up to isometry. Let $K\coloneqq \IQ[x]/(P_{10}(x))$ and $k\coloneqq \IQ[x]/(R_{10}(x))$, where $R_{10}$ is the unique polynomial of degree $5$ such that 
\[
x^5 R_{10}(x + x^{-1})=P_{10}(x)
\]
(cf. \cite[p.~194]{mcmullen.k3.minimum}). 
Observe that $k$ is a totally real field, since all roots $t_0,\ldots,t_4$ of $R_{10}$ are real. 
Among these, only one root, say \(t_0\), is greater than $2$. Lehmer's polynomial $P_{10}$ has exactly two real roots, namely $\lambda_{10}$ and $\lambda_{10}^{-1}$, which satisfy $\lambda_{10} + \lambda_{10}^{-1} = t_0$. It follows that the four Archimedean places $v_1,\ldots,v_4$ of $k$ corresponding to $t_1,\ldots,t_4$ ramify in $K$. Since the extension $K/k$ of degree \(2\) is unramified at all finite places (see, e.g., \cite[Proposition~3.1]{gross.mcmullen}) and the class number of \(K\) is \(1\), we have an exact sequence
\[
\mathcal{O}_K^\times \xlongrightarrow{N_{K/k}} \mathcal{O}_k^\times \xlongrightarrow{A} \{\pm 1\}^4,
\]
where $N_{K/k}$ is the norm map and 
\[
A(u) \coloneqq (\operatorname{sgn}(u(t_1)),\ldots,\operatorname{sgn}(u(t_4))),
\]
where we are viewing $u\in \mathcal{O}_k^\times$ as a polynomial in $k$. 
Indeed, for a unit $u\in \mathcal{O}_k^\times$, the sign $\operatorname{sgn}(u(t_i))\in \{\pm 1\}$ is the local norm residue symbol at the Archimedean place $v_i$ corresponding to $t_i$ (cf. \cite[Theorem~V.1.3]{neukirch}). 
Since the extension $K/k$ is ramified only at $v_1,\ldots,v_4$, it follows that $A(u) = (1,1,1,1)$ if and only if $u\in N_{K/k}(K^\times)$ (see \cite[Corollary~VI.5.8]{neukirch}).
Assume that $u=N_{K/k}(\bar{u})$ for some $\bar{u}\in K^\times$. 
Since $K$ has class number~$1$, the fractional ideal $(\bar{u})$ can be written by Hilbert 90 as 
\[
(\bar{u}) = \left(u'\cdot \chi(u')^{-1}\right)
\]
for some $u'\in K^\times$, where $\chi$ is the generator of $\operatorname{Gal}(K/k)$.
Hence, $\bar{u} = \bar{u}' \cdot u'\cdot \chi(u')^{-1}$ for some $\bar{u}'\in \mathcal{O}_K^\times$. Taking norms, we obtain $u = N_{K/k}(\bar{u}) = N_{K/k}(\bar{u}')$, as desired. 

By \cite[Theorem~5.2]{mcmullen.k3.minimum}, any $P_{10}(x)$-lattice is a twist $L_0(u)$ of the principal lattice $L_0$, which is isometric to $U^{\oplus 5}$ by \cite[Theorem~8.5]{McMullen:K3}. 
Assume that $L_0(u)$ is unimodular. Then $u\in \mathcal{O}_k$ is a unit, and two twists $L_0(u)$, $L_0(u')$ are isometric whenever $u^{-1}u'\in N_{K/k}(\mathcal{O}_K^\times)$ \cite[p.~192]{mcmullen.k3.minimum}. 
In particular, the tuple 
\[
(\varepsilon_1,\ldots,\varepsilon_4) = A(u)\in \{\pm 1 \}^4
\]
determines the isometry class of $L_0(u)$.

Since exactly two of the values $R_{10}'(t_1),\ldots,R_{10}'(t_4)$ are positive, say $R_{10}'(t_1)$ and $R_{10}'(t_2)$, \cite[Theorem~8.3]{McMullen:K3} implies that the signature of $L_0(u)$ is $(1,9)$ if and only if $u(t_1),u(t_2)<0$ and $u(t_3),u(t_4)>0$. 
Thus, the only units $u\in \mathcal{O}_k^\times$ that yield a twist $L_0(u)$ of signature $(1,9)$ are those satisfying 
\[
A(u) = (-1,-1,1,1). 
\]
Therefore, every such twist $L_0(u)$ is the unique unimodular $P_{10}(x)$-lattice of signature $(1,9)$ up to isometry.
\end{proof}

In \Cref{sec: existence}, we gave an example of a blow-up $Z_1$ of $\mathbb{P}^2$ in $10$ points lying on a cuspidal cubic curve with an automorphism of dynamical degree $\lambda_{10}$. The next result says that this is the unique such surface:

\begin{theorem} \label{lem: uniqueness}
Let ${\rm char}(k) = 2$.
Let $\widetilde{Y}$ be a blow-up of $\mathbb{P}^2$ at $10$ distinct points such that $|{-K_{\widetilde{Y}}}| = \{E\}$, where $E$ is an irreducible cuspidal curve of genus $1$. Assume that there exists an automorphism $\tau \in \Aut(\widetilde{Y})$ with $h(\tau) = {\rm log}~\lambda_{10}$. Then, the following hold:
\begin{enumerate}
\item The surface $\widetilde{Y}$ contains no $(-2)$-curves.
\item The group $\Aut(\widetilde{Y})$ is an extension of $\mathbb{Z}/31\mathbb{Z}$ by $W_{\Enriqueslattice}(2)$.
\item The surface $\widetilde{Y}$ can be defined over $\mathbb{F}_2$.
\item There is an isomorphism $\widetilde{Y} \cong Z_1$, where $Z_1$ is as in \Cref{figure:diagram}.
\item The conjugacy class of $\tau \in \Aut(\widetilde{Y})$ acts on $\Pic^0(E) \cong k$ as multiplication by a root of Lehmer's polynomial $P_{10}$ \eqref{eq:Lehmers.polynomial}. Conversely, for every root $\alpha \in k$ of $P_{10}$, there exists a unique conjugacy class of $\tau \in \Aut(\widetilde{Y})$ with $h(\tau) = {\rm log}~\lambda_{10}$ and such that $\tau$ acts on $\Pic^0(E) \cong k$ as multiplication by $\alpha$.
\end{enumerate}
\end{theorem}

\begin{proof}
Pick $10$ disjoint $(-1)$-curves $E_1,\hdots,E_{10}$ such that the contraction of the $E_i$ yields a birational morphism $\pi\colon \widetilde{Y} \to \mathbb{P}^2$ and let $H$ be the strict transform of a general line in $\mathbb{P}^2$.
The divisors $H,E_1,\hdots,E_{10}$ define an isometry $\mathbb{Z}^{1,10} \cong \Pic(\widetilde{Y})$. We have 
$
-K_{\widetilde{Y}} \sim 3H - \sum_{i=1}^{10} E_i
$
and $K_{\widetilde{Y}}^{\perp} \cong \Enriqueslattice$. Since $E$ is anti-canonical, we also have a restriction homomorphism
\[
\varphi \colon K_{\widetilde{Y}}^{\perp}  \longrightarrow  \Pic^0(E).
\]

 For Claim (1), assume seeking a contradiction that $\widetilde{Y}$ contains a $(-2)$-curve. By adjunction, any $(-2)$-curve is orthogonal to $K_{\widetilde{Y}}$.
 Denote by $\Delta \subseteq K_{\widetilde{Y}}^{\perp} \cong \Enriqueslattice$ the sublattice generated by classes of $(-2)$-curves. Then, $\tau$ preserves the chain of sublattices
 \[
    2K_{\widetilde{Y}}^{\perp} \subsetneq \Delta + 2K_{\widetilde{Y}}^{\perp} \subsetneq K_{\widetilde{Y}}^{\perp}.
 \]
 Here, the first inclusion is strict since $\Delta \neq 0$ and $(-2)$-classes are not $2$-divisible. To see that the second inclusion is also strict, observe that $2K_{\widetilde{Y}}^{\perp}  \subseteq \Ker(\varphi)$ as $\Pic^0(E)\cong k$ is $2$-torsion, that $\Delta \subseteq \Ker(\varphi)$, and that the image of $\varphi$ is non-trivial because $\varphi(E_i-E_j) \neq 0$ for $i \neq j$, as we blow up distinct points. 

The action of $\tau$ on $K_{\widetilde{Y}}^{\perp}/2K_{\widetilde{Y}}^{\perp}$ has characteristic polynomial given by \Cref{eq:Lehmer.mod2}, so $(\Delta + 2K_{\widetilde{Y}}^{\perp})/2K_{\widetilde{Y}}^{\perp}$ is one of the two $5$-dimensional subspaces of 
\[
K_{\widetilde{Y}}^{\perp}/2K_{\widetilde{Y}}^{\perp} \cong \Enriqueslattice/2\Enriqueslattice
\]
invariant under the isometry of order $31$ induced by $\tau$. 
By \cite[Section~1.4]{BarthPeters}, the orthogonal group of $\Enriqueslattice/2\Enriqueslattice$ equipped with the quadratic form $\frac{1}{2}q \pmod{2\IZ}$, $q$ being the quadratic form on $\Enriqueslattice$, is ${\rm GO}^+_{10}(2)$ in ATLAS \cite{atlas} notation, and $\tau$ lies in the simple subgroup ${\rm O}^+_{10}(2) \subseteq {\rm GO}^+_{10}(2)$ of index $2$ of isometries of quasi-determinant~$1$. By \cite[p.~180]{atlas}, there are two ${\rm O}^+_{10}(2)$-conjugacy classes of maximal isotropic subspaces of $\Enriqueslattice/2\Enriqueslattice$ with $2295 \equiv 1 \pmod{31}$ members each, so $\tau$ preserves one of each family. Moreover, all maximal isotropic subspaces are conjugate under ${\rm GO}^+_{10}(2)$ so that, in summary, there is a unique isometry class of lattices between $2\Enriqueslattice$ and $\Enriqueslattice$ that is preserved by $\tau$. As the $2$-elementary lattice $\Enriqueslattice(2)$ has an isometry of dynamical degree $\lambda_{10}$ that extends to an isometry of $\Enriqueslattice$, it is an example of such a lattice, so we conclude that $\Delta + 2K_{\widetilde{Y}}^{\perp} \cong \Enriqueslattice(2)$. But $\Enriqueslattice(2)$ has no $(-2)$-vectors, a contradiction.

We now proceed with the remaining claims. For this, recall first that by a result of Vinberg \cite[Theorem 2.2]{Allcock}, the Weyl group $W_{\Enriqueslattice}$ has index $2$ in the orthogonal group of $\Enriqueslattice$ and in fact $\rm{O}(\Enriqueslattice) = W_{\Enriqueslattice} \times \{\pm 1\}$. In particular, we can consider $W_{\Enriqueslattice}$ as the subgroup of $\operatorname{O}(\Pic(\widetilde{Y}))$ that fixes $K_{\widetilde{Y}}$ and preserves the positive cone. Thus, the representation of $\Aut(\widetilde{Y})$ on $\Pic(\widetilde{Y})$ factors through $W_{\Enriqueslattice}$. We claim that we may assume that $\tau^* \in W_{\Enriqueslattice}$ is the inverse of the standard Coxeter element $w$ (compare \cite[Section 8]{McMullen:dynamics.blowups.proj.plane}) that acts on $H,E_1,\hdots,E_{10}$ as
\begin{align*}
w(H)    & \sim  2H - E_2-E_3-E_4,   & w(E_3)  & \sim  H - E_2 - E_3, \\
w(E_1)  & \sim  H - E_3 -E_4,        & w(E_n)  & \sim E_{n+1} \quad \text{ for } 4 \leq n \leq 9,  \\
w(E_2)  & \sim H - E_2 - E_4, & w(E_{10}) & \sim E_1. 
\end{align*}
To see this, note that by \cite[p. 154]{McMullen:coxeter}, the characteristic polynomial of $w$ is Lehmer's polynomial \(P_{10}\) (\ref{eq:Lehmers.polynomial}). From \Cref{lem: E10} and since $\rm{O}(\Enriqueslattice) = W_{\Enriqueslattice} \times \{\pm 1\}$, we conclude that $w = (w')^{-1} \circ  \tau^*\circ w'$ for some $w' \in W_{\Enriqueslattice}$. By \cite[Lemma~2.9]{Harbourne:blowing}, the collection $w'(E_1),\hdots,w'(E_{10})$ is another exceptional configuration for a blow-down to $\mathbb{P}^2$. Thus, after replacing $H,E_1,\hdots,E_{10}$ by $w'(H),w'(E_1),\hdots,w'(E_{10})$, we may assume that $(\tau^*)^{-1}$ is of the above form.

Now, consider the blow-down $\pi\colon \widetilde{Y} \to \mathbb{P}^2$ of the $E_i$ with $p_i \coloneqq \pi(E_i)$. We choose coordinates such that $\pi(E)$ is the cuspidal cubic $C =\{ y^2 z = x^3\}$, so that the unique flex point of $C$ is $q \coloneqq [0:1:0]$, and such that $p_1 = [1:1:1]$. Via the parametrization $\psi\colon k \to C^{\mathrm{sm}}(k)$, $t \mapsto [t:1:t^3]$, of the smooth locus $C^{\mathrm{sm}}$ of $C$, addition on $k$ is identified with the group law on $C^{\mathrm{sm}}(k) \cong \Pic^0(C)$, $p \mapsto \mathcal{O}(p-q)$.

The automorphism $\tau$ induces an automorphism of $C^{\mathrm{sm}}$ which we can write with respect to the parametrization above as $\tau_C(t) = \alpha t + \beta$ for some $\alpha \in k^{\times}$ and $\beta \in k$. Observe that $\alpha$ is exactly the image of $\tau$ under the natural map $\Aut(\widetilde{Y}) \to \Aut(\Pic^0(C)) \cong k^{\times}$, because the tangent space $T_q C$ of $C$ at its flex point $q$ is identified with the tangent space of $\Pic^0_C$ under the identification $p \mapsto \mathcal{O}(p-q)$ and $\tau$ acts on $T_q C$ as multiplication by $\alpha$. In particular, $\alpha$ is a root of Lehmer's polynomial.

Note that $\tau_C^{-n}(1) = \alpha^{-n}(1 + \sum_{i=0}^{n-1} \alpha^i \beta)$. Thus, we have
\begin{align*}
p_{n} &=  \psi (\alpha^{n-11} (1 + \sum_{i=0}^{10-n} \alpha^i \beta ) )   \qquad \text{ for } n = 4,\ldots,10 \\
\{p_3\} &= (\ell_{\tau(p_1)p_4} \cap C) - \{\tau(p_1),p_4\} = \{\psi(\alpha +\beta + \alpha^{-7}(1 + \sum_{i=0}^{6} \alpha^i \beta))\} 
\\
\{p_2\} &= (\ell_{\tau(p_3)p_3} \cap C)  - \{\tau(p_3),p_3 \} 
  = \{\psi(\alpha + \alpha^2 + \alpha^{-6} + \alpha^{-7} + \alpha \beta + \beta + \alpha^{-7} \beta) \} \\
\{p_2\} &= (\ell_{\tau(p_2)p_4} \cap C) - \{\tau(p_2),p_4\} =\\
&= \{\psi(\alpha^2 + \alpha^3 + \alpha^{-5} + \alpha^{-6} + \alpha^{-7} + (\alpha^{-7} + \sum_{i=0}^7{\alpha^{i-5}} )\beta)\}
\end{align*}
where $\ell_{\tau(p_i)p_j}$ is the line through $\tau(p_i) = \pi(w(E_i))$ and $p_j$. 
Now, for $\alpha \in k$ a root of Lehmer's polynomial, the sum
$$
\alpha^{-5}+\alpha^{-4}+\alpha^{-3}+\alpha^{-2}+\alpha^{-1}+\alpha^2
$$
of the coefficients of $\beta$ in the two expressions of $p_2$ is non-zero, so that $\beta$ is uniquely determined by $\alpha$. 
In other words, the scalar $\alpha$ uniquely determines the points $p_i$, and hence the surface $\widetilde{Y}$.

Next, we prove Claim (2). First, note that $\Aut(\widetilde{Y})$ acts faithfully on $\Pic(\widetilde{Y})$, since every automorphism in the kernel preserves the curves $E_i$ and descends to an automorphism of $\mathbb{P}^2$, but the only automorphism of $\mathbb{P}^2$ fixing the $p_i$ is the identity. Then, by \cite[Lemma 3.6]{Harbourne}, we have a short exact sequence
\begin{equation} \label{eq: harbourne sequence}
0 \longrightarrow W_{\Enriqueslattice}(2) \longrightarrow \Aut(\widetilde{Y}) \longrightarrow G \longrightarrow 0,
\end{equation}
where $G \subseteq \Aut(\Pic^0(E))$ is the group of automorphisms for which there exists an isometry of $K_{\widetilde{Y}}^{\perp}$ making $\varphi$ equivariant. 
Since $\Aut(\Pic^0(E)) \cong k^{\times}$, we deduce that $G$ is cyclic. 
Moreover, as the image of $\tau$ in $G$ is a root of Lehmer's polynomial, and hence a primitive $31$-st root of unity, we have $31 \mid |G|$.
On the other hand, as $\Aut(\widetilde{Y}) \subseteq W_{\Enriqueslattice}$, we have
\[
G \subseteq W_{\Enriqueslattice}/W_{\Enriqueslattice}(2) \cong {\rm GO}^+_{10}(2).
\]
By \cite[p. 142]{atlas}, the centralizer of an element of order $31$ in ${\rm GO}^+_{10}(2)$ has order $31$, hence $|G| = 31$, as desired.

To prove Claim (3), recall that, because the cohomological dimension of a finite field is $1$, the surface $\widetilde{Y}$ is defined over $\mathbb{F}_2$ if and only if $\widetilde{Y}$ and its Frobenius pullback $\widetilde{Y}^{(2)}$ are isomorphic over $k$. Now, if $\tau$ is an automorphism of dynamical degree $\lambda_{10}$ on $\widetilde{Y}$ acting through a root $\alpha$ of Lehmer's polynomial on $\Pic^0(C)$, then its Frobenius pullback $\tau^{(2)}$ is an automorphism of dynamical degree $\lambda_{10}$ on $\widetilde{Y}^{(2)}$ that acts on the corresponding Jacobian as $\alpha^2$. Since we have proved above that this scalar uniquely determines the surface, it suffices to show that there is an automorphism $\tau' \in \Aut(\widetilde{Y})$ of dynamical degree $\lambda_{10}$ that acts on $\Pic^0(C)$ as multiplication by $\alpha^2$. This follows from the exact sequence \eqref{eq: harbourne sequence}: indeed, by \cite[p. 142]{atlas}, there exists an element $w \in W_{\Enriqueslattice}$ such that that the image $\overline{w}$ of $w$ in ${\rm GO}^+_{10}(2)$ normalizes $G = \langle \overline{\tau} \rangle$ and such that $\overline{w}^{-1} \overline{\tau} \overline{w} = \overline{\tau}^2$. Since the kernel of $W_{\Enriqueslattice} \to G$ is contained in $\Aut(\widetilde{Y})$, we have thus found $ \tau' \coloneqq w^{-1} \tau w \in \Aut(\widetilde{Y})$ such that $\tau'$ has the same characteristic polynomial as $\tau$ and $\tau'$ acts on $\Pic^0(C)$ as multiplication by $\alpha^2$, as desired. The same argument also proves Claim (4), as all ten possible choices of the scalar $\alpha$, namely $\{\alpha^{\pm 2^i}\}$ for $i=0,\ldots 4$, are realized on $\widetilde{Y}$, hence $\widetilde{Y}$ is unique and thus isomorphic to the surface $Z_1$ of \Cref{figure:diagram}.

Finally, for Claim (5), it suffices to show that if $\tau,\tau' \in \Aut(\widetilde{Y})$ have dynamical degree $\lambda_{10}$ and if they act by multiplication by the same $\alpha$ on $\Pic^0(C)$, then they are conjugate in $\Aut(\widetilde{Y})$. 
By \Cref{lem: E10}, there exists $w \in W_{E_{10}}$ with $w^{-1} \tau w = \tau'$. Since $\overline{\tau} = \overline{\tau'}$, the image $\overline{w}$ of $w$ in ${\rm GO}^+_{10}(2)$ lies in the centralizer of $\overline{\tau}$. By \cite[p.~142]{atlas}, this centralizer is the subgroup generated by $\overline{\tau}$, that is, $\overline{w} \in G$, and so $w \in \Aut(\widetilde{Y})$ by the exact sequence \eqref{eq: harbourne sequence}.
\end{proof}

The unique rational surface of \Cref{lem: uniqueness} will play the role of the canonical cover of $\XLehmer$ in the proof of the following result, which will imply \Cref{thm: uniqueness intro} (2) and \Cref{thm:automorphisms} of the introduction:

\begin{theorem} \label{thm: uniqueness main}
If $X$ is an Enriques surface over an algebraically closed field $k$ with $\operatorname{char}(k) = 2$ and $\sigma \in \Aut(X)$ is an automorphism with $h(\sigma) = \log \lambda_{10}$, then the following hold:
\begin{enumerate}
    \item The surface $X$ contains no $(-2)$-curves.
    \item The group $\Aut(X)$ is an extension of $\mathbb{Z}/31\mathbb{Z}$ by $W_{\Enriqueslattice}(2)$.
    \item The surface $X$ can be defined over $\mathbb{F}_2$.
    \item There are ten conjugacy classes of elements of dynamical degree $\lambda_{10}$ in $\Aut(X)$.
    \item There is an isomorphism $X \cong \XLehmer$.
\end{enumerate}
\end{theorem}
\begin{proof}
Let $\pi \colon Y \to X$ be the canonical cover of $X$, let $\gamma \colon \widetilde{Y} \to Y$ be the minimal resolution of $Y$, let $\sigma \in \Aut(X)$ be an automorphism of dynamical degree $\lambda_{10}$, and let $\tau$ be its lift to $\widetilde{Y}$. 
By \Cref{prop: char2nonnormal,prop: char2normal}, $Y$ is a normal rational surface with a unique elliptic singularity, so by \cite[end of Section 13]{Schroeer:Enriques.normal.K3.coverings}, the morphism $\pi$ is an $\alpha_2$-torsor.
By \cite[Section 3]{Matsumoto}, the exceptional divisor $E$ of $\gamma$ is an integral cuspidal curve of genus \(1\) and self-intersection~$-1$. 
Since $K_Y \sim 0$, we have $K_{\widetilde{Y}} \sim_{\mathbb{Q}} \mu E$ for some $\mu \in \mathbb{Q}$. By adjunction,
\[
    0 = \deg_{E}(K_{\widetilde{Y}} + E) = (1 + \mu)E^2,
\]
so $\mu = -1$, that is, $E$ is an anti-canonical curve on $\widetilde{Y}$.
Moreover, we have
\[
    \Pic(\widetilde{Y}) = \Pic(Y) \oplus \mathbb{Z}\cdot E \cong \Pic(X) \oplus \mathbb{Z} \cdot E,
\]
compatibly with the $\tau$ and $\sigma$ actions, so $\tau$ has dynamical degree $\lambda_{10}$ on $\widetilde{Y}$.

By \Cref{lem: uniqueness}, the surface $\widetilde{Y}$ is unique. Note that $\gamma \colon \widetilde{Y} \to Y$ identifies $\Aut(Y)$ with $\Aut(\widetilde{Y})$: indeed, every automorphism of $\widetilde{Y}$ descends to $Y$, since $Y$ is the contraction of the unique anti-canonical curve, and conversely, all automorphisms of $Y$ lift to $\widetilde{Y}$, because minimal resolutions are unique.
By \cite[Proposition 3.1]{Martin}, we have an exact sequence of group schemes of the form
\begin{equation} \label{eq: normalizer}
1 \longrightarrow \alpha_2 \longrightarrow N_{\alpha_2} \longrightarrow \Aut_X,
\end{equation}
where $N_{\alpha_2} \subseteq \Aut_{Y}$ is the normalizer of the $\alpha_2$-action corresponding to $\pi \colon Y \to X$.
As explained before \Cref{prop: char2ordinary}, the morphism $N_{\alpha_2} \to \Aut_X$ is surjective on $k$-points, since $\pi$ is the canonical cover of $X$. In particular, $\Aut(X) \cong N_{\alpha_2}(k)\subseteq \Aut(Y)$.

Now, Claim (1) follows from \cite[Proposition 8.8, proof of Theorem 14.1]{Schroeer:Enriques.normal.K3.coverings} or from \Cref{lem: uniqueness} and the fact that every $(-2)$-curve on $X$ would have preimage supported on a $(-2)$-curve on $\widetilde{Y}$.

For Claim (2), recall that $W_{\Enriqueslattice}(2) \subseteq \Aut(X)$ by \cite[Theorem 1.1]{Allcock}. Since we are assuming the existence of $\sigma$, we deduce that $\Aut(X)\cong N_{\alpha_2}(k)$ strictly contains $W_{\Enriqueslattice}(2)$. But $\Aut(Y) \cong \Aut(\widetilde{Y})$ is an extension of $\mathbb{Z}/31\mathbb{Z}$ by $W_{\Enriqueslattice}(2)$ by \Cref{lem: uniqueness}, hence contains $W_{\Enriqueslattice}(2)$ as maximal subgroup, and so $\Aut(X)\cong N_{\alpha_2}(k) = \Aut(Y)$. 

Claims (3), (4) and (5) will follow immediately from \Cref{lem: uniqueness} if we can prove that the above $\alpha_2$-action on $Y$ is the unique one such that $N_{\alpha_2}(k) = \Aut(Y)$. Using the correspondence between group actions of height $1$ and restricted Lie subalgebras of $H^0(Y,T_Y)$ (see, e.g., \cite[Exp.~VIIA, Théorème~7.2]{SGA3}), we thus need to show that there is a unique line $\ell_{\rm Enr} \subseteq H^0(Y,T_Y)$ such that for all $D \in \ell_{\rm Enr}$ we have $D^2 = 0$, $D$ has no fixed points (in the sense of \cite[Section 2.2]{Matsumoto}), and such that for all $\psi \in \Aut(Y)$, there exists $\lambda \in k$ such that $\psi D \psi^{-1} = \lambda D$.

To this end, we use that by \cite[Corollary 3.7]{Ekedahl.Hyland.Shepherd-Barron} or \cite[Theorem 1.4, Proposition 3.7]{Matsumoto}, the tangent sheaf $T_Y$ is isomorphic to $\mathcal{O}_Y^{\oplus 2}$, and all its global sections $D$ satisfy $D^2 = 0$. We can thus consider the $2$-dimensional conjugation representation
\[
\rho \colon \Aut(Y) \longrightarrow {\rm GL}(H^0(Y,T_Y))
\]
and we have to show that there is a unique $1$-dimensional $\rho(\Aut(Y))$-stable subspace of $H^0(Y,T_Y)$ consisting of fixed point free derivations. 

By \cite[Proposition 3.7]{Matsumoto}, there is a line $\ell_{\rm rat} \subseteq H^0(Y,T_Y)$ parametrizing the derivations with fixed points. On the other hand, from \Cref{thm: existence main}, we know that there is some fixed point free derivation $D$ spanning a $\rho(\Aut(Y))$-stable subspace, corresponding to the Enriques quotient $\XLehmer$. Thus, the representation $\rho$ is a direct sum of two $1$-dimensional representations $\rho_1$, $\rho_2$ and we have to show that $\rho_1 \neq \rho_2$.
Using the fact that $W_{\Enriqueslattice}(2)$ is generated by involutions, that $\rho$ is a direct sum of two $1$-dimensional representations, and that $\Aut(k) =k^{\times}$ contains no elements of order $2$, we have $W_{\Enriqueslattice}(2) \subseteq \Ker(\rho)$. In particular, we have $\rho_1 \neq \rho_2$ if and only if the two eigenvalues $\lambda_1$ and $\lambda_2$ of the $\tau$-action on $H^0(Y,T_Y)$ are distinct.

We compute these eigenvalues using our example $X = \XLehmer$ in \Cref{thm: existence main}, taking $\tau$ to be the lift of $\sigmaLehmer$. Since the Lie algebra of $\Aut_Y$ is abelian by \cite[Theorem 1.4]{Matsumoto}, the tangent space of $N_{\alpha_2}$ has dimension \(2\). Moreover, $h^0(X,T_X) = 1$ by \cite[Corollary 1.4.9]{CossecDolgachevLiedtke}, so the map $$H^0(Y,T_Y) \to H^0(X,T_X)$$ is surjective. Thus, one of the eigenvalues of $\rho(\tau)$, say $\lambda_1$, is the eigenvalue of conjugation by $\sigmaLehmer$ on the $1$-dimensional space $H^0(X,T_{X})$. By \Cref{rmk: Giacomos remark}, we have $\lambda_1 = \zeta^8$.

On the other hand, since $T_Y \cong \mathcal{O}_Y^{\oplus 2}$, the global sections of $T_Y$ generate the tangent space at every smooth point of $Y$, so the determinant of the conjugation action on $H^0(Y,T_Y)$ can be identified with the pullback action on $H^0(Y,\omega_Y)^{\vee} \cong H^2(Y,\mathcal{O}_Y)$. 
Since $\widetilde{Y}$ is smooth and rational, we have
\[
H^1(\widetilde{Y}, \mathcal{O}_{\widetilde{Y}}) = H^2(\widetilde{Y}, \mathcal{O}_{\widetilde{Y}}) = 0.
\]
It follows from the Leray spectral sequence associated to $\gamma \colon \widetilde{Y} \to Y$, together with the theorem on formal functions, 
that there are natural isomorphisms
\[
H^2(Y, \mathcal{O}_Y) \cong H^0\!\left(Y, R^1\gamma_* \mathcal{O}_{\widetilde{Y}}\right) 
\cong H^1(E, \mathcal{O}_E).
\]
Thus, we conclude that the automorphism 
\[
    \tau|_E^* \colon H^1(E,\mathcal{O}_E) \longrightarrow H^1(E,\mathcal{O}_E)
\]
is scaling by $\lambda_1\lambda_2$.

As $H^1(E,\mathcal{O}_E)$ is naturally isomorphic to the tangent space of $\Pic^0_E$ at the identity, the scalar $\lambda_1 \lambda_2$ is nothing but the scalar $\alpha$ that appears in \Cref{lem: uniqueness} (5). Now, if $\lambda_1 = \lambda_2$, then $\alpha = \zeta^{16}$, which is not a root of Lehmer's polynomial $P_{10}$. So, we must have $\lambda_1 \neq \lambda_2$, which concludes the proof.
\end{proof}

\bibliographystyle{amsplain}
\bibliography{Enriques}

\end{document}